\documentclass{amsart}
\usepackage{mathrsfs}

\theoremstyle{plain}
\newtheorem{thm}{Theorem}[section]

\theoremstyle{definition}
\newtheorem{definition}[thm]{Definition}
\newtheorem{remark}[thm]{Remark}

\def\bull{{\hbox{\bf .}}}
\def\nbull{{\kern.8pt\hbox{\bf .}\kern.8pt}}
\def\newdot{{\kern.8pt\cdot\kern.8pt}}

\def\,{\relax\ifmmode\mskip\thinmuskip\else\thinspace\fi}
\def\{{\relax\ifmmode\lbrace\else $\lbrace$\fi}
\def\}{\relax\ifmmode\rbrace\else $\rbrace$\fi}

\newcommand\E{\mathbb{E}}
\newcommand\N{\mathbb{N}}
\newcommand\R{\mathbb{R}}
\renewcommand\P{\mathbb{P}}

\newcommand{\SD}{{\mathscr D}}

\newcommand{\SW}{{\mathscr W}}

\def\mathpal#1{\mathop{\mathchoice{\text{\rm #1}}%
   {\text{\rm #1}}{\text{\rm #1}}%
   {\text{\rm #1}}}\nolimits}

\def\Ric{\mathpal{Ric}}
\def\grad{\mathpal{grad}}

\def\grad{\mathop{\rm grad}\nolimits}
\def\di{\displaystyle}
\def\f{\frac}
\def\a{\alpha }
\def\b{\beta }
\def\D{\Delta }
\def\d{\delta }

\def\G{\Gamma }
\def\g{\gamma }

\def\n{\nabla }

\def\om{\omega }

\begin{document}

\title[Horizontal diffusion in path space] {Horizontal diffusion in $C^1$ path
  space}

\author[M. Arnaudon]{Marc Arnaudon} \address{Laboratoire de Math\'ematiques et
  Applications\hfill\break\indent CNRS: UMR 6086\hfill\break\indent
  Universit\'e de Poitiers, T\'el\'eport 2 - BP 30179\hfill\break\indent
  F--86962 Futuroscope Chasseneuil Cedex, France}
\email{marc.arnaudon@math.univ-poitiers.fr}

\author[A.\,K. Coulibaly ]{Abdoulaye Coulibaly} \address{Laboratoire de
  Math\'ematiques et Applications\hfill\break\indent CNRS: UMR
  6086\hfill\break\indent Universit\'e de Poitiers, T\'el\'eport 2 - BP
  30179\hfill\break\indent F--86962 Futuroscope Chasseneuil Cedex, France}
\email{abdoulaye.coulibaly@math.univ-poitiers.fr}

\author[A. Thalmaier]{Anton Thalmaier} \address{Unit\'e de Recherche en
  Math\'ematiques\hfill\break\indent Universit\'e du Luxembourg, 162A, Avenue
  de la Fa\"\i encerie\hfill\break\indent L--1511 Luxembourg, Grand-Duchy of
  Luxembourg} \email{anton.thalmaier@uni.lu} 
\keywords{Brownian motion, damped parallel transport, horizontal diffusion, 
Monge-Kantorovich problem, Ricci curvature}

%
%

\begin{abstract}\noindent
  We define horizontal diffusion in $C^1$ path space over a Riemannian
  manifold and prove its existence. If the metric on the manifold is
  developing under the forward Ricci flow, horizontal diffusion along Brownian
  motion turns out to be length preserving.  
  As application, we prove contraction
  properties in the Monge-Kantorovich minimization problem for probability
  measures evolving along the heat flow.
  For constant rank diffusions, differentiating a family of coupled diffusions
  gives a derivative process with a covariant derivative of finite variation. 
  This construction provides an alternative method to filtering out redundant noise.
\end{abstract}

\maketitle
\tableofcontents

%
%

\section{Preliminaries}\label{Section1}
\setcounter{equation}0

The main concern of this paper is to answer the following question: Given a
second order differential operator $L$ without constant term on a manifold $M$ and
a $C^1$ path $u\mapsto \varphi(u)$ taking values in $M$, is it possible to
construct a one parameter family $X_t(u)$ of diffusions with generator $L$ and
starting point $X_0(u)=\varphi(u)$, such that the derivative with respect to
$u$ is locally uniformly bounded?  
If the manifold is $\R^n$ and the
generator $L$ a constant coefficient differential operator, there is an obvious solution: 
the family $X_t(u)=\varphi(u)+Y_t$, where $Y_t$ is an $L$-diffusion starting at $0$, 
has the required properties. 
But already on $\R^n$ with a non-constant generator, the question becomes difficult. 

In this paper we give a
positive answer for elliptic operators $L$ on general manifolds; the result also covers 
time-dependent elliptic generators $L=L(t)$. 
It turns out that the constructed family of diffusions solves the
ordinary differential equation in the space of semimartingales:
\begin{equation}
 \label{ODE}
\partial_uX_t(u)=W(X(u))_t(\dot\varphi(u)),
\end{equation}
where $W(X(u))$ is the so-called deformed parallel translation along the
semimartingale $X(u)$. 

The problem is similar to finding flows associated to
derivative processes as studied in \cite{Cruzeiro:83,
Cipriano-Cruzeiro:99, Cipriano-Cruzeiro:05, Driver:92,
Hsu:95, Gong-Zhang:07} and \cite{Fang-Luo:08}. However it is
transversal in the sense that in these papers diffusions with the same starting
point are deformed along a drift which vanishes at time~$0$. In contrast,
we want to move the starting point but to keep the generator. 
Our strategy of proof consists in iterating parallel couplings for closer and
closer diffusions. In the limit, the solution may be considered as an infinite
number of infinitesimally coupled diffusions. We call it horizontal
$L$-diffusion in $C^1$ path space.

If the generator $L$ is degenerate, we are able to solve~\eqref{ODE} only in the
constant rank case; by parallel coupling we construct a family of diffusions
satisfying~\eqref{ODE} at $u=0$. In particular, the derivative 
of $X_t(u)$ at $u=0$ has finite variation compared to parallel transport. 

Note that our construction requires only a connection on the fiber bundle 
generated by the ``carr\'e du champ'' operator. In the previous approach of
\cite{Elworthy-LeJan-Li:99}, a stochastic differential equation is needed and
$\n$ has to be the Le Jan-Watanabe connection associated to the SDE.

The construction of families of $L(t)$-diffusions $X_\bull(u)$ with $\partial_u
X_\bull(u)$ locally uniformly bounded has a variety of applications.  In 
Stochastic Analysis, for instance, it allows to deduce Bismut type formulas without 
filtering redundant noise. If only the derivative with respect to~$u$ at
$u=0$ is needed, parallel coupling as constructed
in~\cite{Arnaudon-Thalmaier-Wang:06} would be a sufficient tool. The horizontal
diffusion however is much more intrinsic by yielding a flow with 
the deformed parallel translation as derivative, 
well-suited to applications in the analysis of path space. 
Moreover for any $u$, the diffusion $X_\bull(u)$ generates the
same filtration as $X_\bull(0)$, and has the
same lifetime if the manifold is complete.

In Section~\ref{Section4} we use the horizontal diffusion to establish a
contraction property for the Monge-Kantorovich optimal transport between
probability measures evolving under the heat flow. 
We only assume that the cost function is a non-decreasing function of distance. 
This includes all Wasserstein distances with respect to the time-dependent Riemannian metric 
generated by the symbol of the generator $L(t)$. 
For a generator which is independent of time, the proof could be achieved using
simple parallel coupling. 
The time-dependent case however requires horizontal diffusion as a tool.

\section{Horizontal diffusion on $C^1$ path space}\label{Section2}
\setcounter{equation}0

Let $M$ be a complete Riemannian manifold with $\rho$ its Riemannian distance. 
The Levi-Civita connection on $M$ will be denoted by $\n$. 

Given a continuous semimartingale $X$ taking values in $M$, we denote by $d^\n X=dX$ its It\^o
differential and by $d_mX$ the martingale part of $dX$. In local coordinates,
\begin{equation}
\label{Itodiff}
d^\n X\equiv 
dX=\left(dX^i+\frac12\,\Gamma_{jk}^i(X)\,d{<}X^j,X^k{>}\right)\frac{\partial}{\partial x^i}
\end{equation}
where $\Gamma_{jk}^i$ are the Christoffel symbols of the Levi-Civita
connection on $M$. In addition, if $dX^i=dM^i+dA^i$ where $M^i$ is a local martingale
and $A^i$ a finite variation process, then 
$$
d_mX=dM^i\frac{\partial}{\partial x^i}.
$$
Alternatively, if 
$$P_{t}(X)\equiv P_{t}^M(X) : T_{X_0}M\to T_{X_t}M$$ denotes parallel translation 
along $X$, then
$$
dX_t=P_{t}(X)\,d\left(\int_0^\bull P_{s}(X)^{-1} \d X_s\right)_t
$$
and
$$
d_mX_t=P_{t}(X)\,dN_t
$$
where $N_t$ is the martingale part of the Stratonovich integral $\int_0^t
P(X)_{s}^{-1} \d X_s$.

If $X$ is a diffusion with generator $L$, we denote by $W(X)$ the so-called deformed
parallel translation along $X$. Recall that $W(X)_t$ is a linear map $T_{X_0}M\to
T_{X_t}M$, determined by the initial condition $W(X)_0={\rm Id}_{T_{X_0}M}$
together with the covariant It\^o stochastic differential equation:
\begin{equation}
\label{E03} DW(X)_t=-\f12\Ric^\sharp(W(X)_t)\, dt+\nabla_{W(X)_t}Z\,dt.
\end{equation}
By definition we have
\begin{equation}
\label{covder}
DW(X)_t=P_t(X)d\left(P_{\bull}(X)^{-1}W(X)\right)_t.
\end{equation}

Note that the It\^o differential~\eqref{Itodiff} and the parallel translation
require only a connection~$\n$ on~$M$. For the deformed parallel
translation~\eqref{E03} however the connection has to be adapted to a metric. 

In this Section the connection and the metric are independent of time. 
We shall see in Section~\ref{Section3} how these notions can be extended 
to time-dependent connections and metrics.

\begin{thm}
\label{T01}
Let\/ $\R\to M$, $u\mapsto \varphi(u)$, be a $C^1$ path in $M$ and let
$Z$ be a vector field on $M$.
Further let $X^0$ be
a diffusion with generator $L=\D/2+Z$, starting at~$\varphi(0)$, and lifetime $\xi$.  
There exists a unique
family $$u\mapsto (X_t(u))_{t\in[0,\xi[}$$ of diffusions with generator $L$,
almost surely continuous in $(t,u)$ and $C^{1}$ in $u$, satisfying $X(0)=X^0$,
$X_0(u)=\varphi(u)$ and
\begin{equation}
\label{PX}
\partial_u X_t(u)=W(X(u))_t(\dot\varphi(u)).
\end{equation}
Furthermore, the process $X(u)$ satisfies the It\^o stochastic differential equation
\begin{equation}
\label{PTX} dX_t(u)=P_{0,u}^{X_t(\newdot)}\,d_mX_t^0+Z_{X_t(u)}\,dt,
\end{equation}
where $P_{0,u}^{X_t(\newdot)} : T_{X_t^0}M\to T_{X_t(u)}M$ denotes parallel
transport along the $C^{1}$ curve $$[0,u]\to M,\quad v\mapsto X_t(v).$$
\end{thm}

\begin{definition}
\label{D1}
We call $t\mapsto \left(X_t(u)_{u\in \R}\right)$ the horizontal $L$-diffusion
in $C^1$ path space $C^1(\R,M)$ over $X^0$, starting at $\varphi$.
\end{definition}

\begin{remark}
  Given an elliptic generator $L$, we can always choose a metric $g$ on $M$ such
  that $L=\D/2+Z$ for some vector field $Z$ where $\D$ is the Laplacian with respect to $g$. 
  Assuming that $M$ is complete with respect to
  this metric, the assumptions of Theorem~\ref{T01} are fulfilled. In the
  non-complete case, a similar result holds with the only difference that
  the lifetime of $X_\bull(u)$ then possibly depends on $u$.
\end{remark}

\begin{remark}
  Even if $L=\D/2$, the solution we are looking for is not
  the flow of a Cameron-Martin vector field: firstly the
  starting point here is not fixed and secondly the vector field would have to
  depend on the parameter~$u$. 
  Consequently one cannot apply for instance Theorem~3.2 in~\cite{Hsu:95}. 
  An adaptation of the proof of the cited result
  would be possible, but we prefer to give a proof using
  infinitesimal parallel coupling which is more adapted to our
  situation.
\end{remark}

\begin{proof}[Proof of Theorem~\textup{\ref{T01}}]

Without loss of generality we may restrict ourselves to the case $u\ge 0$.

A. {\it Existence. }
Under the assumption that a solution $X_t(u)$ exists, we have for any stopping time $T$,  
$$
W_{T+t}(X(u))(\dot\varphi(u))=W_t(X_{T+\newdot }(u))\,(\partial X_T(u)),
$$
for $t\in {[0,\xi(\om)-T(\om)[}$ and $\om\in\{T<\xi\}$.
Here $\partial X_T:=(\partial X)_T$ denotes the derivative process $\partial X$ 
with respect to $u$, stopped at the random time $T$;
note that by Eq.~\eqref{PX}, $(\partial X_T)(u)=W(X(u))_T(\dot\varphi(u))$.
Consequently we may localize and replace the time interval $[0,\xi[$ by
$[0,\tau\wedge t_0]$ for some $t_0>0$, where $\tau$ is the first exit time of $X$ from
a relatively compact open subset $U$ of $M$ with smooth boundary.  

We may also assume that $U$ is sufficiently small and included in the domain of a local chart; 
moreover we can choose
$u_0\in{]0,1]}$ with $\int_0^{u_0}\|\dot\varphi(u)\|\,du$ small enough such that 
the processes constructed for $u\in [0,u_0]$ stay in the domain
$U$ of the chart. 
At this point we use the uniform boundedness of $W$ on $[0,\tau\wedge t_0]$.

For $\alpha >0$, we define by induction a family of processes 
$(X^\alpha_t(u))_{t\ge0}$ indexed by $u\ge 0$ as follows: 
$X^\alpha (0)=X^0$,
$X_0^\alpha(u)=\varphi(u)$, and if $u\in {]n\alpha ,(n+1)\alpha ]}$ for some
integer $n\ge 0$, $X^\alpha (u)$ satisfies the It\^o equation
\begin{equation}
\label{EPC}
dX_t^\alpha (u)=P_{X_t^\alpha (n\alpha ), X_t^\alpha
(u)}d_mX_t^\alpha (n\alpha )+Z_{X_t^\alpha (u)}\,dt,
\end{equation}
where $P_{x,y}$ denotes parallel translation along the minimal geodesic from
$x$ to $y$. We choose $\alpha$ sufficiently small so that all the minimizing
geodesics are uniquely determined and depend smoothly of the endpoints: since
$X\alpha(u)$ is constructed from $X\alpha(n\alpha)$ via parallel coupling~\eqref{EPC},
there exists a constant $C>0$ such that
\begin{equation}
\label{E27}
\rho(X_t^\a(u),X_t^\a(n\a))\le
\rho(X_0^\a(u),X_0^\a(n\a))\,e^{Ct}\le \|\dot\varphi\|_\infty\,
\a e^{Ct_0}
\end{equation}
(see e.g. \cite{Kendall:86}).

The process $\partial X^\alpha(u)$ satisfies the covariant It\^o stochastic
differential equation
\begin{equation}
\label{E21}
D\partial X^\alpha(u)=\nabla_{\partial X^\alpha(u)}P_{X^\alpha
  (n\alpha ),\nbull}d_mX_t^\alpha (n\alpha )+\nabla_{\partial X^\alpha(u)}Z\,dt -\f12\Ric^\sharp
(\partial X^\alpha(u))\, dt,
\end{equation}
(see \cite{Arnaudon-Thalmaier:03} Eq. (4.7), along with Theorem 2.2). 

\smallbreak
\noindent{\bf Step 1 } We prove that if $X$ and $Y$ are two $L$-diffusions 
stopped at $\tau_0:=\tau\wedge t_0$ and living in $U$, then there exists a constant $C$ such that
\begin{equation}
 \label{E22}
\E\left[\sup_{t\le \tau_0}\left\|W(X)_t-W(Y)_t\right\|^2\right]
\le C\,\E\left[\sup_{t\le \tau_0}\left\|X_t-Y_t\right\|^2\right].
\end{equation}
Here we use the Euclidean norm defined by the chart.

Writing
$$
L=a^{ij}\partial_{ij}+b^j\partial_j
$$
with $a^{ij}=a^{ji}$ for $i,j\in\{1,\ldots,\dim M\}$, and denoting by $(a_{ij})$ the inverse
of $(a^{ij})$, the connection $\nabla'$ with Christoffel symbols
$$
(\Gamma')_{ij}^k=-\f12(a_{ik}+a_{jk})b^k
$$
has the property that all $L$-diffusions are $\nabla'$-martingales.

On the other hand, for $\n'$-martingales $X$ and $Y$ living in $U$, with
$N^X$, respectively $N^Y$, their martingale parts in the chart $U$, It\^o's formula for
$\n'$-martingales yields
\begin{align*}
\big\langle &(N^X)^k-(N^Y)^k, (N^X)^k-(N^Y)^k\big\rangle_t\\
&=  (X_t^k-Y_t^k)^2-(X_0^k-Y_0^k)^2\\
&\quad-2\int_0^t(X_s^k-Y_s^k)\,d((N_s^X)^k-(N_s^Y)^k)\\
&\quad+\int_0^t(X_s^k-Y_s^k)\left((\G')_{ij}^k(X_s)\,d\langle(N^X)^i,(N^X)^j\rangle_s-(\G')_{ij}^k(Y_s)\,d\langle(N^Y)^i,(N^Y)^j\rangle_s\right).
\end{align*}
From there we easily prove that, for $U$ sufficiently small, there exists a
constant $C>0$ such that for all $L$-diffusions $X$ and $Y$ stopped at the
first exit time of $U$,
\begin{equation}
 \label{E24}
\E\left[ \langle N^X-N^Y\vert N^X-N^Y\rangle_{\tau_0}\right]\le C\E\left[\sup_{t\le \tau_0}\left\|X_t-Y_t\right\|^2\right]
\end{equation}
where $\langle N^X-N^Y\vert N^X-N^Y\rangle$ denotes the Riemannian quadratic
variation (see e.g.~\cite{Arnaudon-Thalmaier:98}).

Writing $W(X)=P(X)\left( P(X)^{-1}W(X)\right)$, an easy
calculation shows that in the local chart
\begin{align}
dW(X)=-&\G(X)(dX,W(X))-\f12(d\G)(X)(dX)(dX,W(X))\notag\\
{}+&\f12\G(X)(dX,\G(X)(dX,W(X)))-\f12\Ric^\sharp(W(X))\,dt+\n_{W(X)}Z\,
  dt.\label{E25}
\end{align}
We are going to use Eq.~\eqref{E25} to evaluate the difference $W(Y)-W(X)$. 
Along with the already established bound
\eqref{E24}, taking into account that $W(X)$, $W(Y)$ and the derivatives of the brackets
of $X$ and $Y$ are bounded in $U$, we can get a bound for 
$F(t):=\E\left[\sup_{s\le t\wedge
    \tau}\|W(Y)-W(X)\|^2\right]$. 
First an estimate of the type
$$
F(t)\le C_1\,\E\left[\sup_{s\le
    \tau_0}\left\|X_s-Y_s\right\|^2\right]+C_2\int_0^tF(s)\,ds,\quad0\le t\le t_0
$$
is derived which then by Gronwall's lemma leads to
\begin{equation}
 \label{E26}
F(t)\le C_1\,e^{C_2t}\,\E\left[\sup_{t\le \tau_0}\left\|X_t-Y_t\right\|^2\right].
\end{equation}
Letting $t=t_0$ in \eqref{E26} we obtain the desired bound \eqref{E22}.

\smallbreak
\noindent{\bf Step 2 } We prove that there exists $C>0$ such that for all $u\in [0,u_0]$, 
\begin{equation}
 \label{E23}
\E\left[\sup_{t\le \tau_0}\rho^2\left(X_t^\a(u),X_t^{\a'}(u)\right)\right]\le C(\a+\a')^2.
\end{equation}

From the covariant equation \eqref{E21} for $\partial X_t^\alpha(v)$ and the 
definition of deformed parallel translation \eqref{E03}, 
\begin{equation*}
DW(X)^{-1}_t=\f12\Ric^\sharp(W(X)^{-1}_t)\, dt-\nabla_{W(X)^{-1}_t}Z\,dt,
\end{equation*}
we have for $(t,v)\in[0,\tau_0]\times [0,u_0]$,
\begin{align*}
 W(X^\a(v))_t^{-1}\, \partial X_t^\alpha(v)&=\dot\varphi(v)+\int_0^tW(X^\a(v))_s^{-1}\n_{\partial
    X_s^\a(v)}P_{X_s^\a(v_\a),\nbull}d_mX_s^\a(v_\a), 
\end{align*}
or equivalently,
\begin{align}
  \partial X_t^\alpha(v)&=W(X^\a(v))_t\,\dot\varphi(v)\notag\\
&\quad+W(X^\a(v))_t\int_0^tW(X^\a(v))_s^{-1}\,\n_{\partial
    X_s^\a(v)}P_{X_s^\a(v_\a),\nbull}d_mX_s^\a(v_\a) \label{E35}
\end{align}
with $v_\a=n\a$, where the integer $n$ is determined by $n\a<v\le
(n+1)\a$.  Consequently, we have
\begin{align*}
  &\rho(X_t^\a(u),X_t^{\a'}(u))\\&=\int_0^u\left\langle d\rho,\left(\partial X_t^\alpha(v),\partial X_t^{\a'}(v)\right)\right\rangle \,dv\\
  &=\int_0^u\left\langle d\rho,\left(W(X^\a(v))_t\dot\varphi(v),W(X^{\a'}(v))_t\dot\varphi(v)\right)\right\rangle \,dv\\
  &+\int_0^u\left\langle d\rho,\left(
      W(X^\a(v))_t\int_0^tW(X^\a(v))_s^{-1}\n_{\partial
        X_s^\a(v)}P_{X_s^\a(v_\a),\nbull}d_mX_s^\a(v_\a)
      ,0\right)\right\rangle dv\\
  &+\int_0^u\left\langle d\rho,\left(0,
      W(X^{\a'}(v))_t\!\int_0^tW(X^{\a'}(v))_s^{-1}\n_{\partial
        X_s^{\a'}(v)}P_{X_s^{\a'}(v_{\a'}),\nbull}d_mX_s^{\a'}(v_{a'})\!
    \right)\right\rangle dv.
\end{align*}
This yields, by means of boundedness of $d\rho$ and deformed parallel translation,
together with~\eqref{E26} and the Burkholder-Davis-Gundy inequalities,
\begin{align*}
  \E\left[\sup_{t\le \tau_0}\rho^2\left(X_t^\a(u),X_t^{\a'}(u)\right)\right]
  &\le\  C\int_0^u\E\left[\sup_{t\le \tau_0}\rho^2\left(X_t^\a(v),X_t^{\a'}(v)\right)\right]\,dv\\
  &\quad+C\int_0^u\E\left[\int_0^{\tau_0}\left\|\n_{\partial X_s^\a(v)}P_{X_s^\a(v_\a),\nbull}\right\|^2\,ds\right]\,dv\\
  &\quad+C\int_0^u\E\left[\int_0^{\tau_0}\left\|\n_{\partial
        X_s^{\a'}(v)}P_{X_s^{\a'}(v_{\a'}),\nbull}\right\|^2\,ds\right]\,dv.
\end{align*}
From here we obtain
\begin{align*}
  \E\left[\sup_{t\le \tau_0}\rho^2\left(X_t^\a(u),X_t^{\a'}(u)\right)\right]
  &\le\ C\int_0^u\E\left[\sup_{t\le \tau_0}\rho^2\left(X_t^\a(v),X_t^{\a'}(v)\right)\right]\,dv\\
  &\quad+C\a^2\int_0^u\E\left[\int_0^{\tau_0}\left\|\partial X_s^\a(v)\right\|^2\,ds\right]\,dv\\
  &\quad+C{\a'}^2\int_0^u\E\left[\int_0^{\tau_0}\left\|\partial
      X_s^{\a'}(v)\right\|^2\,ds\right]\,dv,
\end{align*}
where we used the fact that for $v\in T_xM$, $\n_vP_{x,\nbull}=0$, together
with $$\rho(X_s^\b(v),X_s^\b(v_\b))\le C\b,\quad\b=\a,\a',$$
see estimate \eqref{E27}.

Now, by Eq.~\eqref{E21} for $D\partial X^\b$, there exists a
constant $C'>0$ such that for all $v\in [0,u_0]$,
$$
\E\left[\int_0^{\tau_0}\left\|\partial X_s^\b(v)\right\|^2\,ds\right]<C'.
$$
Consequently,
\begin{align*}
\E\left[\sup_{t\le
    \tau_0}\rho^2\left(X_t^\a(u),X_t^{\a'}(u)\right)\right]&\le C\int_0^u\E\left[\sup_{t\le
    \tau_0}\rho^2\left(X_t^\a(v),X_t^{\a'}(v)\right)\right]\,dv\\&\quad+2CC'(\a+\a')^2
\end{align*}
which by Gronwall lemma yields
$$
\E\left[\sup_{t\le
    \tau_0}\rho^2\left(X_t^\a(u),X_t^{\a'}(u)\right)\right]\le C\,(\a+\a')^2
$$
for some constant $C>0$. This is the desired inequality.

\smallbreak
\noindent{\bf Step 3 } 
From inequality \eqref{E23} we deduce that there exists a limiting process
$$(X_t(u))_{0\le t\le \tau_0,\ 0\le u\le u_0}$$
such that for all $u\in
[0,u_0]$ and $\a>0$,
\begin{equation}
 \label{E28}
\E\left[\sup_{t\le \tau_0}\rho^2\left(X_t^\a(u),X_t(u)\right)\right]\le C\a^2.
\end{equation}

In other words, for any fixed $u\in [0,u_0]$, the process $(X_t^\a(u))_{t\in [0,\tau_0]}$
converges to $(X_t(u))_{t\in [0,\tau_0]}$ uniformly in $L^2$ as~$\a$ tends
to~$0$. Since these processes are $\n'$-martingales, convergence also
holds in the topology of semimartingales. This implies in particular that for any $u\in
[0,u_0]$, the process $(X_t(u))_{t\in [0,\tau_0]}$ is a diffusion with generator $L$,
stopped at $\tau_0$. 

Extracting a subsequence $(\a_k)_{k\ge 0}$ convergent to
$0$, we may assume that almost surely, for all dyadic $u\in [0,u_0]$, $$\sup_{t\le
  \tau_0}\rho\left(X_t^\a(u),X_t(u)\right)$$ converges to $0$. 
Moreover we can choose $(\a_k)_{k\ge 0}$ of the 
form $\a_k={2^{-n_k}}$ with $(n_k)_{k\ge 0}$ an increasing sequence of positive integers. 
Due to~\eqref{E27}, we can
take a version of the processes $(t,u)\mapsto X_t^{\a_k}(u)$ such that
$$u\mapsto X_t^{\a_k}(u)$$ is uniformly Lipschitz in $u\in \N \a_k\cap [0,u_0]$ with a Lipschitz
constant independent of~$k$ and~$t$.  
Passing to the limit, we obtain that a.s for any $t\in [0,\tau_0]$, the map $u\mapsto X_t(u)$ 
is uniformly Lipschitz in $u\in \SD\cap [0,u_0]$ with a Lipschitz
constant independent of~$t$, where $\SD$ is the set of dyadic numbers. 
Finally we can choose a version of $(t,u)\mapsto X_t(u)$ which is a.s.~continuous 
in $(t,u)\in [0,\tau_0]\times [0,u_0]$, and hence uniformly Lipschitz in $u\in [0,u_0]$.

\smallbreak
\noindent{\bf Step 4 } We prove that almost surely,  
$X_t(u)$ is differentiable in $u$ with derivative $W(X(u))_t(\dot\varphi(u))$. 
More precisely, we show that in local coordinates, almost surely, for all $t\in [0,\tau_0]$, $u\in [0,u_0]$, 
\begin{equation}
 \label{E29}
X_t(u)=X_t^0+\int_0^u W(X(v))_t(\dot\varphi(v))\,dv.
\end{equation}

From the construction it is clear that almost surely, for all $t\in
[0,\tau_0]$, $u\in [0,u_0]$,
\begin{align*}
  X_t^{\a_k}&(u)=X_t^0+\int_0^u W(X^{\a_k}(v))_t(\dot\varphi(v))\,dv\\
  &+\int_0^u \left(W(X^{\a_k}(v))_t\int_0^t W(X^{\a_k}(v))_s^{-1}\n_{\partial
      X_s^{\a_k}(v)}P_{X_s^{\a_k}(v_{\a_k}),\nbull}d_mX_s^{\a_k}(v_{\a_k})\right)dv.
\end{align*}
This yields
\begin{align*}
  X_t&(u)-X_t^0-\int_0^uW(X(v))_t(\dot\varphi(v))\,dv\\
  &=X_t(u)-X_t^{\a_k}(u)+\int_0^u\left(W(X^{\a_k}(v))_t-W(X(v))_t\right)\dot\varphi(v)\,dv\\
  &\quad+\int_0^u \left(W(X^{\a_k}(v))_t\int_0^t W(X^{\a_k}(v))_s^{-1}\n_{\partial
      X_s^{\a_k}(v)}P_{X_s^{\a_k}(v_{\a_k}),\nbull}d_mX_s^{\a_k}(v_{\a_k})\right)dv.
\end{align*}
The terms of right-hand-side are easily estimated, where in the estimates 
the constant $C$ may change from one line to another. 
First observe that
$$
\E\left[\sup_{t\le\tau_0}\left\|X_t(u)-X_t^{\a_k}(u)\right\|^2\right]\le
C\a_k^2.
$$
Using \eqref{E22} and \eqref{E28} we have
\begin{align*}
  \E&\left[\sup_{t\le\tau_0}\left\|
      \int_0^u\left(W(X^{\a_k}(v))_t-W(X(v))_t\right)\,dv
    \right\|^2\right]\\
  &\quad\le \E\left[\sup_{t\le\tau_0}
    \int_0^u\left\|W(X^{\a_k}(v))_t-W(X(v))_t\right\|^2\,dv
  \right]\\
  &\quad=\int_0^u\E\left[\sup_{t\le\tau_0}
    \left\|W(X^{\a_k}(v))_t-W(X(v))_t\right\|^2
  \right]\,dv\le C\a_k^2,
\end{align*}
and finally
\begin{align*}
  &\E\left[\sup_{t\le\tau_0}\left\| \int_0^u \left(W(X^{\a_k}(v))_t\int_0^t
        W(X^{\a_k}(v))_s^{-1}\n_{\partial
          X_s^{\a_k}(v)}P_{X_s^{\a_k}(v_{\a_k}),\nbull}d_mX_s^{\a_k}(v_{\a_k})\right)dv
    \right\|^2\right]\\
  &\le C\int_0^u\E\left[ \sup_{t\le\tau_0}\left\|\int_0^t
      W(X^{\a_k}(v))_s^{-1}\n_{\partial
        X_s^{\a_k}(v)}P_{X_s^{\a_k}(v_{\a_k}),\nbull}d_mX_s^{\a_k}(v_{\a_k})\right\|^2
  \right]dv\\
  &\le C\a_k^2\int_0^u\E\left[\int_0^{\tau_0}\left\|\partial X_s^{\a_k}(v)\right\|^2\,ds\right]dv\le C\a_k^2.
\end{align*}
We deduce that
$$
\E\left[\sup_{t\le\tau_0}\left\| X_t(u)-X_t^0-\int_0^u
    W(X(v))_t(\dot\varphi(v))\,dv\right\|^2\right]\le C\a_k^2.
$$
Since this is true for any $\a_k$, using continuity in $u$ of $X_t(u)$, we
finally get almost surely for all $t,u$,
$$
X_t(u)=X_t^0+\int_0^u W(X(v))_t(\dot\varphi(v))\,dv.
$$

\smallbreak
\noindent{\bf Step 5 }
Finally we are able to prove Eq.~\eqref{PTX}: $$
dX_t(u)=P_{0,u}^{X_t(\newdot)}\,d_mX_t^0+Z_{X_t(u)}\,dt.$$

Since a.s.~the mapping $(t,u)\mapsto \partial X_t(u)$ is continuous, the map
$u\mapsto \partial X(u)$ is continuous in the topology of uniform convergence
in probability. We want to prove that $u\mapsto \partial X(u)$ is continuous
in the topology of semimartingales.  

Since for a given connection on a
manifold, the topology of uniform convergence in probability and the topology
of semimartingale coincide on the set of martingales (Proposition 2.10 of
\cite{Arnaudon-Thalmaier:98}), it is sufficient to find a connection on $TM$
for which $\partial X(u)$ is a martingale for any $u$.  
Again we can localize in the domain of a chart. Recall that for all $u$, 
the process $X(u)$ is a
$\nabla'$-martingale where $\n'$ is defined in step~1. 
Then by~\cite{Arnaudon:97}, Theorem~3.3, this implies that the derivative with respect
to $u$ with values in $TM$, denoted here by $\partial X(u)$, is a
$(\nabla')^c$-martingale with respect to the complete lift $(\nabla')^c$ of
$\nabla'$. This proves that $u\mapsto\partial X(u)$ is continuous in the
topology of semimartingales.

\begin{remark}
  Alternatively, one could have used that given a generator $L'$, the topologies
  of uniform convergence in probability on compact sets and the topology of
  semimartingales coincide on the set of $L'$-diffusions. Since 
  the processes $\partial X(u)$ are diffusions with the same generator, 
  the result could be derived as well.
\end{remark}

As a consequence, we have formally
\begin{equation}
\label{E1} D\partial X=\nabla_u dX-\frac12 R(\partial X,dX)dX.
\end{equation}
Since
$$
dX(u)\otimes dX(u)=g^{-1}(X(u))\,dt
$$
where $g$ is the metric tensor, Eq.~\eqref{E1} becomes
$$
D\partial X=\nabla_u dX-\f12\Ric^\sharp(\partial X)\,dt.
$$
On the other hand, Eq.~\eqref{PX} and Eq.~\eqref{E03} for $W$ yield
$$
D\partial X =-\f12\Ric^\sharp(\partial X)\, dt+\nabla_{\partial X}Z\,dt.
$$
From the last two equations we obtain
$$
\nabla_u dX=\nabla_{\partial X}Z\,dt.
$$
This along with the original equation
$$
dX^0 =d_mX^0+Z_{X^0}\,dt
$$
gives
$$
dX_t(u)=P_{0,u}^{X_t(\newdot)}\,d_mX_t^0+Z_{X_t(u)}\,dt,
$$
where $$P_{0,u}^{X_t(\newdot)}\colon\, T_{X_t}M\to T_{X_t(u)}M$$ denotes parallel
transport along the $C^1$ curve $v\mapsto X_t(v)$.

\smallbreak
B. {\it Uniqueness. }
Again we may localize in the domain of a chart $U$.  Letting $X(u)$ and $Y(u)$
be two solutions of Eq.~\eqref{PX}, then for $(t,u)\in
[0,\tau_0[\times [0,u_0]$ we find
in local coordinates,
\begin{equation}
 \label{E33}
Y_t(u)-X_t(u)=\int_0^u\big(W(Y(v))_t-W(X(v))_t\big)(\dot\varphi(v))\,dv.
\end{equation}
On the other hand, using \eqref{E22} we have
\begin{equation}
 \label{E34}
\E\left[\sup_{t\le \tau_0} \|Y_t(u)-X_t(u)\|^2\right]
\le C\int_0^u\E\left[\sup_{t\le \tau_0}\|Y_t(v)-X_t(v)\|^2\right]\,dv
\end{equation}
from which we deduce that almost surely, for all $t\in [0,\tau_0]$,
$X_t(u)=Y_t(u)$. Consequently, exploiting the fact that the two processes are continuous in
$(t,u)$, they must be indistinguishable.  
\end{proof}

\section{Horizontal diffusion along non-homogeneous diffusion}\label{Section3}
\setcounter{equation}0 

In this Section we assume that the elliptic generator
is a $C^1$ function of time: $L=L(t)$ for $t\ge 0$.  
Let $g(t)$ be the metric on $M$ such that
$$
L(t)=\f12\D^t+Z(t)
$$
where $\D^{t}$ is the $g(t)$-Laplacian and $Z(t)$ a vector field on $M$.

Let $(X_t)$ be an inhomogeneous diffusion with generator $L(t)$. Parallel transport
$P^t(X)_t$ along the $L(t)$-diffusion $X_t$ is defined analogously to
\cite{Arnaudon-Coulibaly-Thalmaier:08} as the linear map
\begin{align*}
  P^t(X)_t\colon T_{X_0}M&\to T_{X_t}M
\end{align*}
which satisfies
\begin{equation}
\label{defPt}
 D^{t}P^t(X)_t=-\f12\,\dot g^\sharp (P^t(X)_t)\,dt
\end{equation}
where $\dot g$ denotes the derivative of $g$ with respect to time; the covariant differential $D^t$ is
defined in local coordinates by the same formulas as $D$, with the only
difference that Christoffel symbols now depend on $t$. 

Alternatively, if $J$ is a
semimartingale over $X$, the covariant differential 
$D^tJ$ may be defined as $\tilde D(0,J)=(0,D^tJ)$,
where $(0,J)$ is a semimartingale along $(t,X_t)$ in $\tilde M=[0,T]\times M$ endowed
with the connection $\tilde \n$ defined as follows: if $$s\mapsto \tilde
\varphi(s)=(f(s),\varphi(s))$$ is a $C^1$ path in $\tilde M$ and $s\mapsto
\tilde u(s)=(\a(s),u(s))\in T\tilde M$ is $C^1$ path over~$\tilde \varphi$,
then
$$
\tilde \n \tilde u(s)=\left(\dot\alpha(s),\big(\n^{f(s)}u\big)(s)\right)
$$
where $\n^t$ denotes the Levi-Civita connection associated to $g(t)$.  It is
proven in \cite{Arnaudon-Coulibaly-Thalmaier:08} that $P^t(X)_t$ is an
isometry from $(T_{X_0}M, g(0,X_0))$ to $(T_{X_t}M, g(t,X_t))$.

The damped parallel translation $W^t(X)_t$ along $X_t$ is the linear map
\begin{align*}
  W^t(X)_t : T_{X_0}M&\to T_{X_t}M
\end{align*}
satisfying
 \begin{equation}
\label{defWt}
 D^{t}W^t(X)_t=\left(\n^t_{W^t(X)_t}Z(t,\newdot)-\f12 (\Ric^t)^\sharp (W^t(X)_t)\right)\,dt.
\end{equation}

If $Z\equiv 0$ and $g(t)$ is solution to the backward Ricci flow:
\begin{equation}
 \label{Rf}
\dot g=\Ric,
\end{equation}
then damped parallel translation coincides with the usual parallel translation:
$$P^{t}(X)=W^{t}(X),$$ (see \cite{Arnaudon-Coulibaly-Thalmaier:08} Theorem~2.3).

The It\^o differential $d^\n Y=d^{\n^t}Y$ of an $M$-valued semimartingale $Y$ is 
defined by formula~\eqref{Itodiff}, with the only
difference that the Christoffel symbols depend on time.

\begin{thm}
\label{T02} Keeping the assumptions of this Section,
let $$\R\to M,\quad u\mapsto \varphi(u),$$ be a $C^1$ path in $M$ and let $X^{0}$
be an $L(t)$-diffusion with starting point $\varphi(0)$ and lifetime $\xi$.
Assume that $(M,g(t))$ is complete for every $t$.  
There exists a
unique family $$u\mapsto (X_t(u))_{t\in[0,\xi[}$$ of $L(t)$-diffusions, which is
a.s.~continuous in $(t,u)$ and $C^{1}$ in $u$, satisfying $$X(0)=X^{0}\ \text{ and }\
X_0(u)=\varphi(u),$$ and solving the equation
\begin{equation}
\label{gPX}
\partial X_t(u)=W^{t}(X(u))_t(\dot\varphi(u)).
\end{equation}
Furthermore, $X(u)$ solves the It\^o stochastic differential equation
\begin{equation}
\label{gPTX} d^{\n}X_t(u)=P_{0,u}^{t,X_t(\newdot)}\,d^{\n(t)}X_t+ Z(t,X_t(u))\, dt,
\end{equation}
where $$P_{0,u}^{t,X_t(\newdot)} : T_{X_t^0}M\to T_{X_t(u)}M$$ denotes 
parallel transport along the $C^{1}$ curve $[0,u]\to M$, $v\mapsto X_t(v)$, 
with respect to the metric $g(t)$.

If $Z\equiv 0$ and if $g(t)$ is given as solution to the backward Ricci flow equation,
then almost surely for all $t$,
\begin{equation}
 \label{isom}
\left\|\partial X_t(u)\right\|_{g(t)}=\left\|\dot\varphi(u)\right\|_{g(0)}.
\end{equation}
\end{thm}

\begin{definition}
\label{D2}
We call $$t\mapsto (X_t(u))_{u\in \R}$$ the horizontal $L(t)$-diffusion in
$C^1$ path space $C^1(\R,M)$ over $X^{0}$, started at $\varphi$.
\end{definition}

\begin{remark}
Eq.~\eqref{isom} says that if $Z\equiv 0$ and if $g$ is solution to the
backward Ricci flow equation, then the horizontal $g(t)$-Brownian motion is length
preserving (with respect to the moving metric).
\end{remark}

\begin{remark}
Again if the manifold $(M,g(t))$ is not necessarily complete for all $t$, a similar
  result holds with the lifetime of $X_\bull(u)$ possibly depending on~$u$.
\end{remark}

\begin{proof}[Proof of Theorem~\textup{\ref{T02}}]
The proof is similar to the one of Theorem~\ref{T01}. 
We restrict ourselves to explaining the differences.

The localization procedure carries over immediately; we work on the time interval
$[0,\tau\wedge t_0]$.
For $\a>0$, we define the approximating process $X_t^{\a}(u)$ by induction as
$$X_t^{\a}(0)=X_t^0,\quad X_0^{\a}(u)=\varphi(u),$$ and if $u\in{]n\a,(n+1)\a]}$ for
some integer $n\ge 0$, then $X^{\a}(u)$ solves the It\^o equation
\begin{equation}
\label{IeXalpha}
d^{\n}X_t^{\a}(u)=P^{t}_{X_t^{\a}(n\a),X_t^{\a}(u)}d_mX_t^{\a}(n\a)+Z(t,X_t(u))\,dt
\end{equation}
where $P_{x,y}^{t}$ is the parallel transport along the minimal geodesic from
$x$ to $y$, for the connection $\n^t$.

Alternatively, letting $\tilde X_t^{\a}=(t,X_t^{\a})$, we may
write~\eqref{IeXalpha} as
\begin{equation}
\label{IetildeXalpha}
d^{\tilde \n}\tilde X_t^{\a}(u)
=\tilde P_{\tilde X_t^{\a}(n\a),\tilde X_t^{\a}(u)}d_m\tilde X_t^{\a}(n\a)+ Z(\tilde X_t^\a(u))\,dt
\end{equation}
where $\tilde P_{\tilde x,\tilde y}$ denotes parallel translation along the
minimal geodesic from $\tilde x$ to $\tilde y$ for the connection $\tilde \n$.

Denoting by $\rho(t,x,y)$ the distance from $x$ to $y$ with respect to the
metric $g(t)$,  It\^o's formula shows that the process $\di
\rho\left(t,X_t^{\a}(u),X_t^{\a}(n\a)\right)$ has locally bounded variation.
Moreover since locally $\partial_t\rho(t,x,y)\le C\rho(t,x,y)$ for $x\not=y$,
we find similarly to~\eqref{E27},
\begin{equation*}
\rho(t, X_t^\a(u),X_t^\a(n\a))\le
\rho(0,X_0^\a(u),X_0^\a(n\a))e^{Ct}\le \|\dot\varphi\|_\infty\,
\a e^{Ct_0}.
\end{equation*}
Since all Riemannian distances are locally equivalent, this implies
\begin{equation}
\label{E27c}
\rho(X_t^\a(u),X_t^\a(n\a))\le
\rho(X_0^\a(u),X_0^\a(n\a))e^{Ct}\le \|\dot\varphi\|_\infty\,
\a e^{Ct_0}
\end{equation}
where $\rho=\rho(0,\newdot,\newdot)$.

Next, differentiating Eq.~\eqref{IetildeXalpha} yields
\begin{align*}
  \tilde D\partial_u \tilde X_t^{\a}(u)
&=\tilde\n_{\partial_u \tilde X_t^{\a}(u)}\tilde P_{\tilde X_t^{\a}(n\a),\nbull}d_m\tilde X_t^{\a}(n\a)\\
  &\quad+\tilde \n_{\partial_u\tilde X_t^\a(u)}Z\,dt -\f12\tilde R\left(\partial_u
    \tilde X_t^{\a}(u),d\tilde X_t^{\a}(u)\right)d\tilde X_t^{\a}(u).
\end{align*}
Using the fact that the first component of $\tilde X_t^{\a}(u)$ has
finite variation, a careful computation of $\tilde R$ leads to the equation
\begin{align*}
  D^{t}\partial_u X_t^{\a}(u)&=\n^t_{\partial_u  X_t^{\a}(u)} P^{t}_{ X_t^{\a}(n\a),\nbull}d_m X_t^{\a}(n\a)\\
  &\quad+
  \n^t_{\partial_uX_t^\a(u)}Z(t,\newdot)-\f12(\Ric^{t})^\sharp\big(\partial_u
    X_t^{\a}(u)\big)\,dt.
\end{align*}

To finish the proof, it is sufficient to remark that in step~1, 
Eq.~\eqref{E24} still holds true for $X$ and $Y$ $g(t)$-Brownian
motions living in a small open set $U$, and that in step~5, the map $u\mapsto \partial X(u)$
is continuous in the topology of semimartingales. This last point is due to the fact that all
$\partial X(u)$ are inhomogeneous diffusions with the same generator, say
$L'$, and the fact that the topology of uniform convergence on compact sets and the topology of
semimartingales coincide on $L'$-diffusions.  
\end{proof}

\section{Application to optimal transport}\label{Section4}
\setcounter{equation}0 
In this Section we assume again that the elliptic generator
$L(t)$ is a $C^1$ function of time with associated metric $g(t)$:
$$
L(t)=\f12\D^t+Z(t)
$$
where $\D^t$ is the Laplacian associated to $g(t)$ and $Z(t)$ is a vector
field. We assume further that for any $t$, the Riemannian manifold $(M,g(t))$ is
metrically complete, and $L(t)$ diffusions have infinite lifetime. 

Letting $\varphi\colon
\R_+\to \R_+$ be a non-decreasing function, we define a cost function
\begin{equation}
 \label{cost}
c(t,x,y)=\varphi(\rho(t,x,y))
\end{equation}
where $\rho(t,\newdot,\newdot)$ denotes distance with respect to $g(t)$. 

To the
cost function $c$ we associate the Monge-Kantorovich minimization between two
probability measures on $M$
\begin{equation}
 \label{Wc}
\SW_{c,t}(\mu,\nu)=\inf_{\eta\in\Pi(\mu,\nu)}\int_{M\times M}c(t,x,y)\,d\eta (x,y)
\end{equation}
where $\Pi(\mu,\nu)$ is the set of all probability measures on $M\times M$
with marginals $\mu$ and $\nu$. 
We denote
\begin{equation}
 \label{Wp}
\SW_{p,t}(\mu,\nu)=\left(\SW_{\rho^p,t}(\mu,\nu)\right)^{1/p}
\end{equation}
the Wasserstein distance associated to $p>0$.

For a probability measure $\mu$ on $M$, the solution of the heat flow equation
associated to $L(t)$ will be denoted by $\mu P_t$.
 
Define a section $(\n^t Z)^\flat\in\G(T^\ast M\odot T^\ast M)$ as follows: for any $x\in M$ and
$u,v\in T_xM$,
$$
(\n^t Z)^\flat(u,v) =\f12 \left(g(t)(\n^t_u
    Z,v)+g(t)(u,\n^t_v Z)\right).
$$
In case the metric does not depend on $t$ and $Z=\grad V$ for some $C^2$
function $V$ on~$M$, then
$$
(\n^t Z)^\flat(u,v)=\n dV(u,v).
$$

\begin{thm}
 \label{WDF} 
We keep notation and assumptions from above.

{\rm a)} Assume
\begin{equation}
 \label{gRb}
\Ric^t-\dot g- 2(\n^t Z)^\flat\ge 0.
\end{equation}
Then the function
\begin{equation*}
t\mapsto \SW_{c,t}(\mu P_t, \nu P_t)
\end{equation*}
is non-increasing.\smallskip

{\rm b)} If for some $k\in\R$,
\begin{equation}
 \label{gRbk}
\Ric^t-\dot g-2(\n^t Z)^\flat\ge kg, 
\end{equation}
then we have for all $p>0$
\begin{equation*}
\SW_{p,t}(\mu P_t,\nu P_t)\le e^{-kt/2}\,\SW_{p,0}(\mu,\nu).
\end{equation*}
\end{thm}

\begin{remark}
Before turning to the proof of Theorem \ref{WDF}, let us mention that in the case $Z=0$, $g$
constant, $p=2$ and $k=0$, item b) is due to \cite{Sturm-vonRenesse:05} and
\cite{Otto-Westdickenberg:05}. In the case where $g$ is a backward Ricci flow
solution, $Z=0$ and $p=2$, statement b) is due to Lott \cite{Lott:07} and
McCann-Topping~\cite{McCann-Topping:07}. 
For extensions about ${\mathcal L}$-transportation, see~\cite{Topping:07}.
\end{remark}

\begin{proof}[Proof of Theorem~\textup{\ref{WDF}}] \ 
  a) Assume that $\Ric^t-\dot g- 2(\n^t Z)^\flat\ge 0$. Then for any
  $L(t)$-diffusion $(X_t)$, we have
\begin{align*}
  d\big(g(t)&(W(X)_t,W(X)_t)\big)\\
  &=\dot g(t)\big(W(X)_t,W(X)_t\big)\,dt+2g(t)\left(D^tW(X)_t,W(X)_t\right)\\
  &=\dot g(t)\big(W(X)_t,W(X)_t\big)\,dt\\
&\quad+2g(t)\left(\n^t_{W(X)_t}Z(t,\newdot)-\f12(\Ric^t)^\sharp(W(X)_t),W(X)_t\right)dt\\
  &=\left(\dot g+2(\n^tZ)^\flat -\Ric^t\right)\big(W(X)_t,W(X)_t\big)\,dt\le 0.
\end{align*}
Consequently, for any $t\ge 0$,
\begin{equation}
 \label{nW}
\|W(X)_t\|_t\le \|W(X)_0\|_0=1.
\end{equation}
For $x,y\in M$, let $u\mapsto \g(x,y)(u)$ be a minimal $g(0)$-geodesic from
$x$ to $y$ in time~$1$: $\g(x,y)(0)=x$ and $\g(x,y)(1)=y$. Denote by
$X^{x,y}(u)$ a horizontal $L(t)$-diffusion with initial condition $\g(x,y)$.

For $\eta\in \Pi(\mu,\nu)$, define the measure $\eta_t$ on $M\times M$ by
$$
\eta_t(A\times B)=\int_{M\times M}\P\big\{X_t^{x,y}(0)\in A,\ X_t^{x,y}(1)\in
B\big\}\,d\eta (x,y),
$$
where $A$ and $B$ are Borel subsets of $M$. 
Then $\eta_t$  has marginals $\mu P_t$ and $\nu P_t$.
Consequently it is sufficient to prove that for any such $\eta$,
\begin{equation}
\label{tbp}
\int_{M\times M}\E\big[c(t,X_t^{x,y}(0),X_t^{x,y}(1))\big]\,d\eta (x,y)
\le \int_{M\times M}c(0,x,y)\,d\eta(x,y).
\end{equation}
On the other hand, we have a.s.,
\begin{align*}
  \rho(t,X_t^{x,y}(0),X_t^{x,y}(1))&\le \int_0^1\big\|\partial_uX_t^{x,y}(u)\big\|_t\,du\\
  &=\int_0^1\big\|W(X^{x,y}(u))_t\,\dot\g(x,y)(u)\big\|_t\,du\\
  &\le \int_0^1\big\|\dot\g(x,y)(u)\big\|_0\,du\\
  &=\rho(0,x,y),
\end{align*}
and this clearly implies
$$
c\big(t,X_t^{x,y}(0),X_t^{x,y}(1)\big)\le c(0,x,y)\quad \hbox{a.s.},
$$
and then \eqref{tbp}.

\smallbreak
b) Under condition \eqref{gRbk}, we have
$$
\f{d}{dt}\,g(t)\big(W(X)_t,W(X)_t\big)\le -k\,g(t)\big(W(X)_t,W(X)_t\big),
$$
which implies
$$
\|W(X)_t\|_t\le e^{-kt/2},
$$
and then
$$
\rho\big(t,X_t^{x,y}(0),X_t^{x,y}(1)\big)\le e^{-kt/2}\rho(0,x,y).
$$
The result follows.
\end{proof}

\section{Derivative process along constant rank diffusion}\label{Section5}
\setcounter{equation}0

In this Section we consider a generator $L$ of constant rank: the image $E$
of the ``carr\'e du champ'' operator $\G(L)\in \G(T M\otimes T M)$ defines a subbundle
of $TM$. In~$E$ we then have an intrinsic metric given by 
$$g(x)=\left(\G(L)|{E(x)}\right)^{-1},\quad x\in M.$$ 
Let $\n$ be a connection on
$E$ with preserves $g$, and denote by $\n'$ the associated semi-connection: if
$U\in \G(TM)$ is a vector field, $\n'_vU$ is defined only if $v\in E$ and
satisfies
$$
\n'_vU=\n_{U_{x_0}}V+[V,U]_{x_0}
$$
where $V\in \G(E)$ is such that $V_{x_0}=v$ 
(see \cite{Elworthy-LeJan-Li:99}, Section~1.3).  We denote by $Z(x)$ the drift of
$L$ with respect to the connection $\n$. 

For the construction of a flow of
$L$-diffusions we will use an extension of $\n$ to $TM$ denoted by $\tilde\n$.
Then the associated semi-connection $\n'$ is the restriction of the classical
adjoint of $\tilde\n$ (see \cite{Elworthy-LeJan-Li:99} Proposition~1.3.1).

\begin{remark}
  It is proven in \cite{Elworthy-LeJan-Li:99} that a connection $\n$ always
  exists, for instance, we may take the Le~Jan-Watanabe connection associated to a well chosen
  vector bundle homomorphism from a trivial bundle $M\times H$ to $E$ where
  $H$ is a Hilbert space.
\end{remark}

If $X_t$ is an $L$-diffusion, the parallel transport $$P(X)_t\colon E_{X_0}\to
E_{X_t}$$ along $X_t$ (with respect to the connection $\tilde \n$) depends only on $\n$. 
The same applies for the It\^o differential $dX_t=d^\n X_t$. 
We still denote by $d_mX_t$ its martingale part. 

We denote by $$\tilde P'(X)_t : T_{X_0}M\to
T_{X_t}M$$ the parallel transport along $X_t$ for the adjoint connection
$(\tilde\n)'$, and by $\tilde D'J$ the covariant differential (with respect to
$(\tilde\n)'$) of a semimartingale $J\in TM$ above $X$; compare \eqref{covder} for
the definition.

\begin{thm}
 \label{T03} We keep the notation and assumptions from above. 
 Let $x_0$ be a fixed point in $M$ and $X_t(x_0)$ an $L$-diffusion starting at~$x_0$. 
For $x\in M$ close to $x_0$, we define the $L$-diffusion $X_t(x)$, started at $x$, by
\begin{equation}
\label{IED}
dX_t(x)=\tilde P_{X_t(x_0),X_t(x)}\,d_mX_t(x_0)+Z(X_t(x))\,dt
\end{equation}
where $\tilde P_{x,y}$ denotes parallel transport (with respect to $\tilde\n$)
along the unique $\tilde \n$-geodesic from $x$ to $y$.  Then
\begin{equation}
\label{DTXd}
\tilde D'T_{x_0}X=\tilde\n_{T_{x_0}X}Z\,dt
-\f12\Ric^\sharp(T_{x_0}X)\,dt
\end{equation}
where $$\Ric^\sharp(u)=\sum_{i=1}^d \tilde R(u,e_i)e_i,\quad u\in T_xM,$$ and
$(e_i)_{i=1,\ldots,d}$ an orthonormal basis of $E_x$ for the metric~$g$.

Under the additional assumption that $Z\in \G(E)$, the differential $\tilde D'T_{x_0}X$ does not
depend on the extension $\tilde\n$, and we have
\begin{equation}
\label{DTXd2}
\tilde D'T_{x_0}X= \n_{T_{x_0}X}Z\,dt
-\f12\Ric^\sharp(T_{x_0}X)\,dt.
\end{equation}
\end{thm}

\begin{proof}
 From \cite{Arnaudon-Thalmaier:03} Eq.~7.4 we have 
\begin{align*}
  \tilde D'T_{x_0}X&=\tilde \n_{T_{x_0}X}\tilde P_{X_t(x_0),\nbull }\,d_mX_t(x_0)+\tilde\n_{T_{x_0}X}Z\,dt\\
  &\quad-\f12\left(\tilde R'(T_{x_0}X,dX(x_0))dX(x_0)+\tilde \n' \tilde T'(dX({x_0}),T_{x_0}X,dX({x_0}))\right)\\
  &\quad-\f12\tilde T'(\tilde D'T_{x_0}X,dX)
\end{align*}
where $\tilde T'$ denotes the torsion tensor of $\tilde\nabla'$. 
Since for all $x\in M$, $\tilde \n_v \tilde P_{{x},\nbull}=0$ if $v\in T_{x}M$, 
the first term in the right vanishes. As a consequence, 
$\tilde D'T_{x_0}X$ has finite variation, and $T'(\tilde D'T_{x_0}X,dX)=0$.
Then using the identity
$$
\tilde R'(v,u)u+\tilde \n' \tilde T'(u,v,u)= \tilde R(v,u)u, \quad u,v\in T_xM,
$$
which is a particular case of identity (C.17) in \cite{Elworthy-LeJan-Li:99},
we obtain
$$
\tilde D'T_{x_0}X=\tilde\n_{T_{x_0}X}Z\,dt-\f12\tilde R(T_{x_0}X,dX(x_0))dX(x_0).
$$
Finally writting 
$$\tilde
R(T_{x_0}X,dX(x_0))dX(x_0)=\Ric^\sharp(T_{x_0}X)\,dt$$ 
yields the result.
\end{proof}

\begin{remark}
  In the non-degenerate case, $\n$ is the Levi-Civita connection associated to
  the metric generated by $L$, and we are in the situation of
  Section~\ref{Section2}. In the degenerate case, in
  general, $\n$ does not extend to a metric connection on $M$. 
  However conditions are given in
  \cite{Elworthy-LeJan-Li:99} (1.3.C) under which $P'(X)$ is adapted to some
  metric, and in this case $T_{x_0}X$ is bounded with respect to the metric.
  
  One would like to extend Theorem~\ref{T01} to degenerate diffusions of
  constant rank, by solving the equation $$\partial_uX(u)=
  \tilde\n_{\partial_uX(u)}Z\,dt -\f12\Ric^\sharp(\partial_uX(u))\,dt.$$
  Our proof does not work in this situation for two reasons. The first
  one is that in general $\tilde P'(X)$ is not adapted to a metric. 
  The second
  one is the lack of an inequality of the type \eqref{E27} since
  $\n$ does not have an extension $\tilde \n$ which is the Levi-Civita
  connection of some metric.
\end{remark}

\begin{remark}
  When $M$ is a Lie group and $L$ is left invariant, then $\tilde \n$ can be
  chosen as the left invariant connection. In this case $(\tilde\n)'$ is the
  right invariant connection, which is metric.
\end{remark}

%
%

\providecommand{\bysame}{\leavevmode\hbox to3em{\hrulefill}\thinspace}

\end{document}